\documentclass[reqno]{amsart} 

\usepackage{amsmath,amssymb}
\usepackage{amsthm}
\usepackage{indentfirst}
\usepackage{graphicx}
\usepackage{xcolor}

\usepackage{graphicx,tikz}

\newtheorem*{thm}{Theorem}

\newtheorem*{lemma}{Lemma}

\theoremstyle{definition}

\theoremstyle{remark}

\begin{document}

\title[]{A Nonlocal Transport Equation modeling Complex \\Roots of Polynomials under Differentiation}

\author[]{Sean O'Rourke}
\address{Department of Mathematics, University of Colorado, Campus Box 395, Boulder, CO 80309-0395, USA}
\email{sean.d.orourke@colorado.edu}

\author[]{Stefan Steinerberger}
\address{Department of Mathematics, University of Washington, Seattle, WA 98195, USA}
\email{steinerb@uw.edu}

\thanks{S.O. has been supported in part by NSF grants ECCS-1610003 and DMS-1810500. S.S. is partially supported by the NSF (DMS-1763179) and the Alfred P. Sloan Foundation. Part of the work was carried out while S.O. was visiting Yale University, he is grateful for the hospitality.}

\begin{abstract} Let $p_n:\mathbb{C} \rightarrow \mathbb{C}$ be a random complex polynomial whose roots are sampled i.i.d. from a radial distribution $2\pi r u(r) dr$ in the complex plane. A natural question
is how the distribution of roots evolves under repeated (say $n/2-$times) differentiation of the polynomial. We conjecture a mean-field expansion for the evolution of $\psi(s) = u(s) s$ 
$$ \frac{\partial \psi}{\partial t} = \frac{\partial}{\partial x} \left( \left( \frac{1}{x} \int_{0}^{x} \psi(s) ds \right)^{-1} \psi(x) \right).$$
The evolution of $\psi(s) \equiv 1$ corresponds to the evolution of random Taylor polynomials 
$$ p_n(z) = \sum_{k=0}^{n}{ \gamma_k \frac{z^k}{k!}} \quad \mbox{where} \quad \gamma_k \sim \mathcal{N}_{\mathbb{C}}(0,1).$$
We discuss some numerical examples suggesting that this particular solution may be stable. We prove that the solution is linearly stable. The linear stability analysis reduces to the classical Hardy integral inequality. Many open problems are discussed.
\end{abstract}

\maketitle

\vspace{-5pt}

\section{Introduction}

\subsection{Introduction.} We ask a simple question, a special case of which is as follows: suppose $p_n:\mathbb{C} \rightarrow \mathbb{C}$ is a random polynomial given by
$$ p_n(z) = \prod_{k=1}^{n}{(z-z_k)} \quad \mbox{where the roots} \quad z_k \sim \mathcal{N}_{\mathbb{C}}(0,1)$$
are independently and identically distributed (i.i.d.) following a standard (complex) Gaussian. What can be said about the roots of the $n/2-$derivative of that polynomial?
Are these roots still distributed like a Gaussian or do they follow another distribution? What if we replace the Gaussian by another
distribution? What if the roots are not i.i.d. but have a density close to that of a nice distribution? To the best of our knowledge, there is only one example where rigorous results are available. We define random Taylor polynomials 
via

$$ p_n(z) = \sum_{k=0}^{n}{ \gamma_k \frac{z^k}{k!}} \quad \mbox{where} \quad \gamma_k \sim \mathcal{N}_{\mathbb{C}}(0,1).$$

Clearly, the $k-$th derivative of a random Taylor polynomial is a random Taylor polynomial of degree $n-k$. In particular, if their roots
happen to have a nice limiting distribution, then this would give us an insight into the possible evolution. A result of Kabluchko \& Zaporozhets
\cite{kab12} shows that the roots (which are not independent) are asymptotically contained in a disk of radius $\sim n$ and the renormalized roots $z_1, \dots, z_n$ converge in distribution to
$$ \frac{1}{n} \sum_{k=1}^{n}{ \delta_{z_k n^{-1}}} \rightarrow \frac{\chi_{|z| \leq 1}}{2\pi |z|} \qquad \mbox{as}~n \rightarrow \infty.$$
This shows that the roots $z_1, \dots, z_n$ have the asymptotic density $\chi_{|z| \leq n}/(2 \pi |z|)$
and that we would expect the roots of the $(t\cdot n)-$derivative, for some $0< t< 1$, to be distributed according to the density
$\chi_{|z| \leq (1-t)n}/(2\pi|z|)$. 
The purpose of this short paper is to pose the question and to present an evolution equation from a mean field approximation; our hope is that this equation might
answer the question posed above for general distributions (and might also be interesting in its own right). 

\begin{center}
\begin{figure}[h!]
\begin{tikzpicture}
\node at (0,0) {\includegraphics[width = 0.5\textwidth]{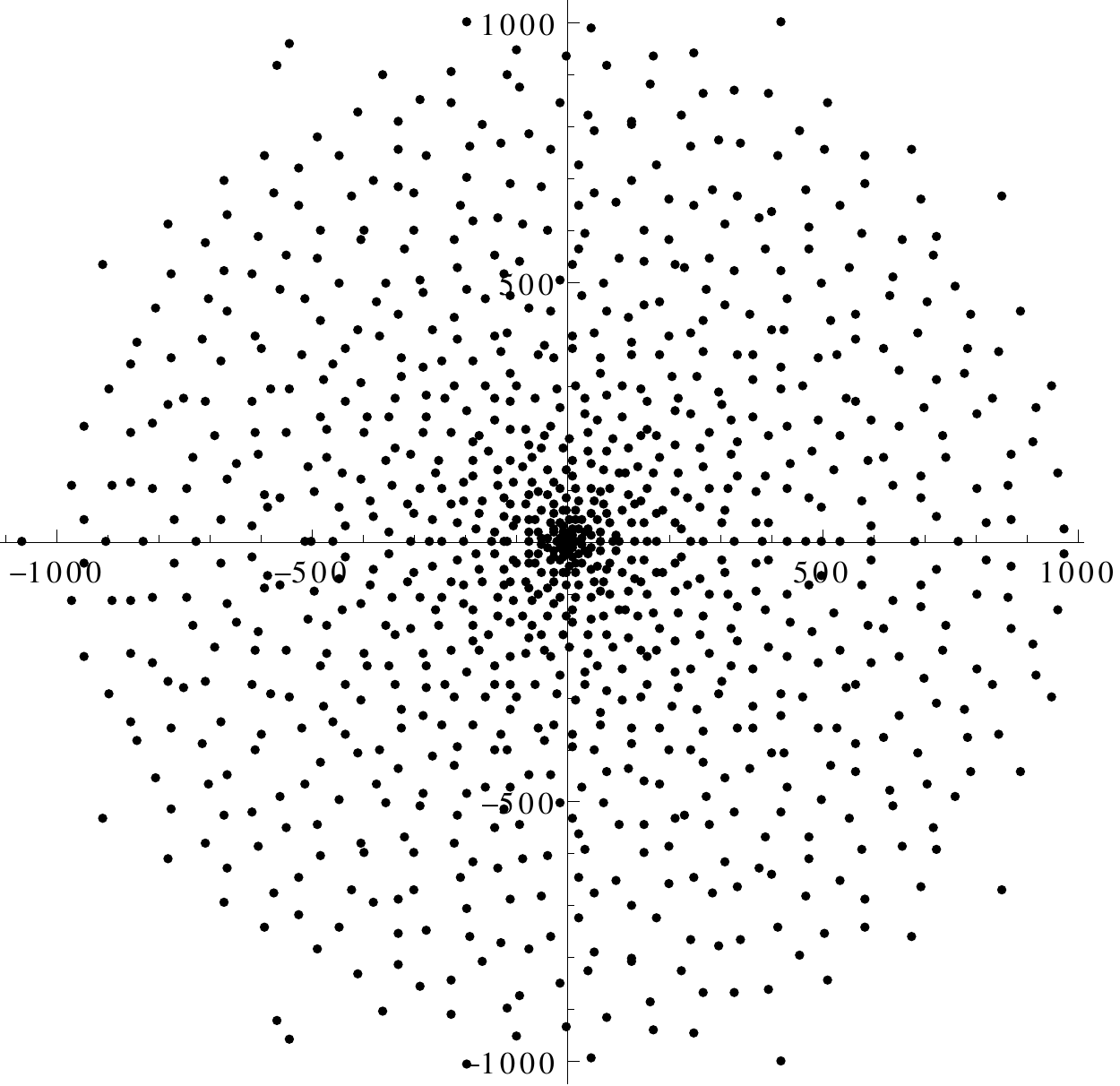}};
\node at (5.5,0) {\includegraphics[width = 0.3\textwidth]{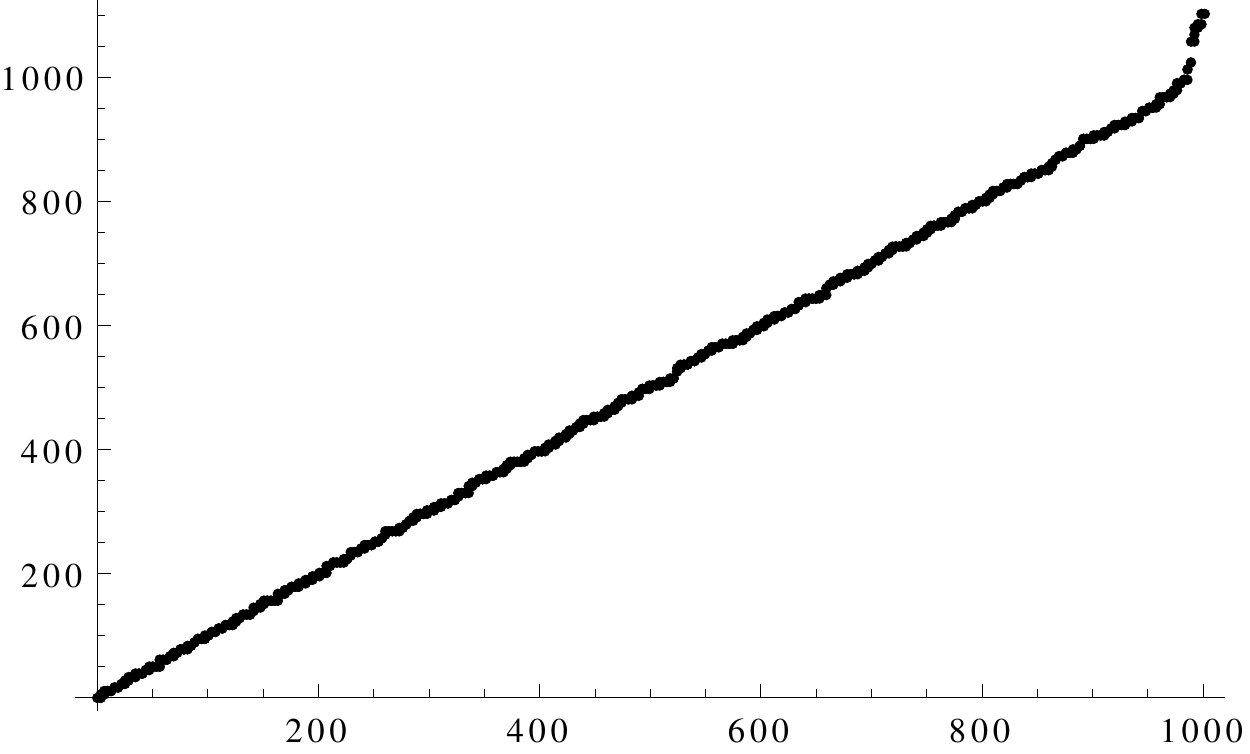}};
\end{tikzpicture}
\caption{Roots of a Random Taylor polynomial of degree 1000 (left). Distances of the roots to the origin (right) are approximately linear as predicted by the limiting measure $1/(2\pi |z|)$.}
\end{figure}
\end{center}

\subsection{Related results.} The detailed study of the distribution of roots of $p_n'$ depending on $p_n$ is an active field \cite{bruj, bruj2, branko, dimitrov, gauss,  han,  kab, lucas, malamud, sean, or, or2, pem, rav, steini2, totik, ull, van}. There are fewer results for the case of repeated differentiation.  If $p_n$ is a polynomial of degree $n$ having $n$ distinct roots on the real line, then the $k-$th derivative has all of its $n-k$ roots also on the real line and one could wonder about their evolution. A result commonly attributed to Riesz \cite{riesz} implies that the minimum gap between consecutive roots of $p_n'$ is bigger than that
of $p_n$: zeroes even out and become more regular. We refer the reader to results of Farmer \& Rhoades \cite{farmer}, Farmer \& Yerrington \cite{farmer2}, Feng \& Yao \cite{feng} and Pemantle \& Subramnian \cite{pem2}. 
Our result is inspired by a one-dimensional investigation due to the second author \cite{steini3}: if the roots of $p_n$ are all real and follow a nice distribution $u(0,x)$,
what can be said about the distribution $u(t,x)$ of the $(t \cdot n)-$th derivative of $p_n$ where $0 < t < 1$? In \cite{steini3} it is proposed that the limiting dynamics exists and is given by
the partial differential equation
$$
 u_t + \frac{1}{\pi} \left(\arctan{ \left( \frac{Hu}{ u}\right)} \right)_x = 0
$$
where the equation is valid on the support $\mbox{supp}~u = \left\{x: u(x) > 0 \right\}$ and $H$ is the Hilbert transform. The equation has been shown to give the correct predictions
for Hermite polynomials and Laguerre polynomials. One interesting aspect is that the equation has similarities to one-dimensional transport
equations that are studied in fluid dynamics, see e.g. \cite{car, castro, cord1, const, cord2, do, dong2, lazar, li, silvestre}. Granero-Belinchon \cite{granero} studied an analogue of the equation on the one-dimensional torus.

\section{Results}
\subsection{The Equation.} Let us assume that $p_n:\mathbb{C} \rightarrow \mathbb{C}$ is a polynomial of degree $n$ whose roots are distributed according to a radial density function
$u(|z|)dz$ on $\mathbb{C}$. Then the density of roots at distance $r$ from the origin is given by $\psi(r) = u(r)r dr$. We will, throughout the paper, understand $u$ as the probability density of the measure
of roots at time $t=0$. Assuming radial structure, understanding $u$ and its evolution in time, is equivalent to understanding $\psi(r) = 2\pi r u(r)$ and its evolution in time (one can be understood as the representation
of the other in polar coordinates). We use $\psi(r)$ to denote the initial distribution and write $\psi(t,r)$ for the distribution of roots after we differentiated $t\cdot n$-times. 
We are interested in the process of differentiation and its effect on the
distribution of the roots. We present a mean field equation that we conjecture models the density of roots at distance $x$ at time $t$ via the nonlocal transport equation
\begin{equation}
\frac{\partial \psi}{\partial t}  = \frac{\partial}{\partial x}\left(  \left(\frac{1}{x} \int_0^x \psi(y) dy\right)^{-1} \psi(x)\right).
\end{equation}
We believe that this equation may be interesting in its own right and give some supporting evidence of this belief. There is a scaling symmetry for $\lambda > 0$
$$ u(t,x) \rightarrow \lambda^{} \cdot u(t,\lambda x).$$
This scaling symmetry is naturally related to the chain rule: if we rescale the roots of the polynomial, then all the derivatives obey the same rescaling as well since
$$ \frac{d}{dz^k} p_n(\lambda z) = \lambda^k  \left( \frac{d }{dz^k} p_n\right)(\lambda z).$$
Needless to say, the factor $\lambda^k$ does not impact the presence or absence of roots implying a scale-invariance property of the system.
There is one more property that can be predicted from the behavior of polynomials: the $(t \cdot n)-$th derivative of a polynomial $p_n$ has $(1-t)n$ roots. In particular,
what one would then expect from any such equation is that there is a constant loss of mass that is independent of the function and independent of time. This loss would
also imply that the function vanishes at time $t=1$. This is indeed the case: if we assume the solutions are continuous, there is a constant loss of mass on $(0, \infty)$: 
\begin{align*}
 \frac{\partial}{\partial t} \int_{0}^{\infty}{ \psi(t,x) dx} &=  \int_{0}^{\infty}{ \frac{\partial}{\partial t} \psi(t,x) dx}\\
 &=   \int_{0}^{\infty}{  \frac{\partial}{\partial x}\left(  \left(\frac{1}{x} \int_0^x \psi(y) dy\right)^{-1} \psi(x)\right) dx}\\
 &= - \lim_{\varepsilon \rightarrow 0}  \left(\frac{1}{\varepsilon} \int_0^{\varepsilon} \psi(y) dy\right)^{-1} \psi(\varepsilon)\\
 &= -1.
\end{align*}

\subsection{Linear Stability.} 
 We note that there is an explicit solution (when properly interpreted) given by
$$ \psi(t,x) = \chi_{0 \leq x \leq 1-t}$$
which corresponds to the random Taylor polynomials and respects their evolution (as implied by the result of Kabluchko \& Zaporozhets \cite{kab12}).

\begin{figure}[h!]
\begin{minipage}[l]{.45\textwidth}
\begin{tikzpicture}
\node at (0,0) {\includegraphics[width = \textwidth]{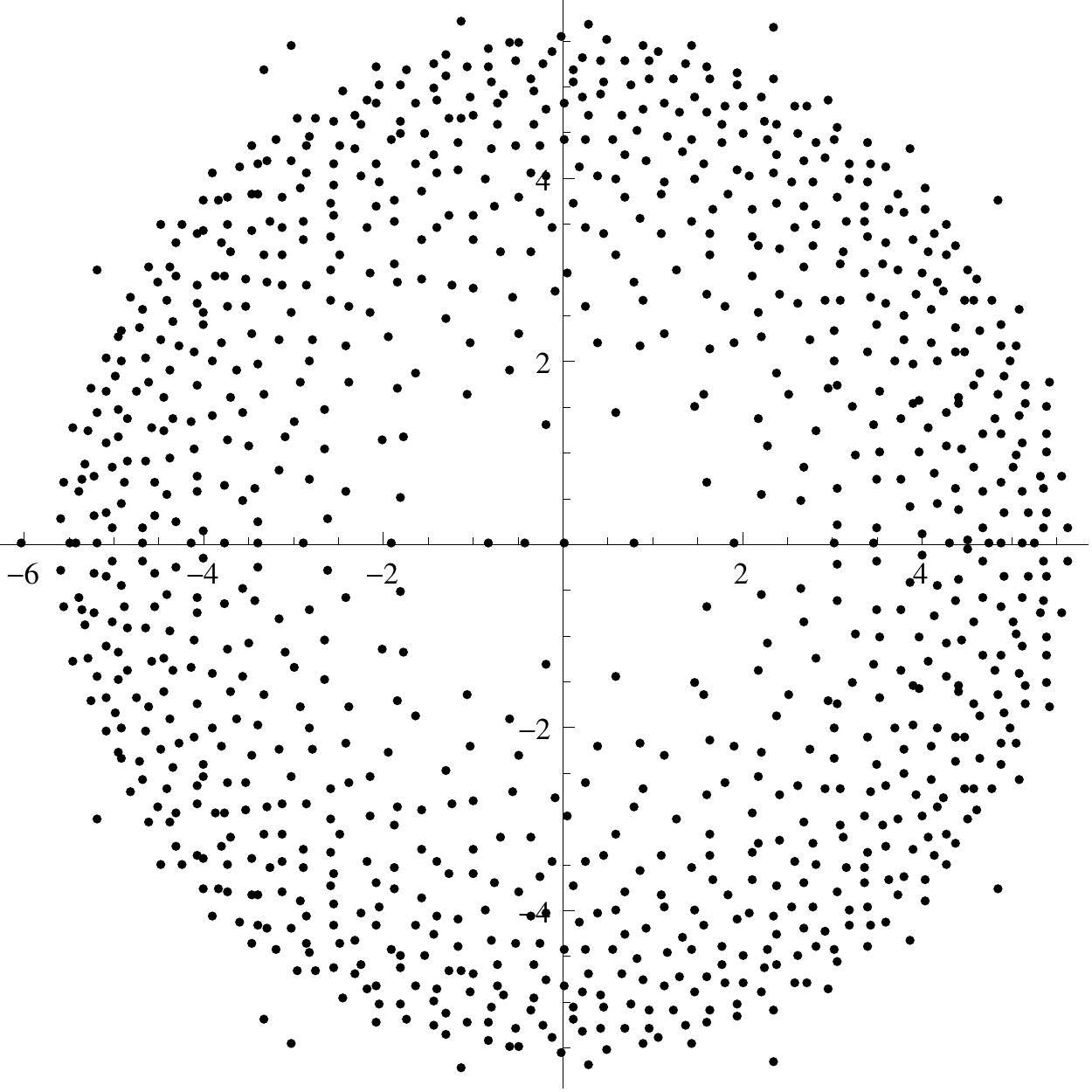}};
\end{tikzpicture}
\end{minipage} 
\begin{minipage}[r]{.45\textwidth}
\begin{tikzpicture}
\node at (0,0) {\includegraphics[width = \textwidth]{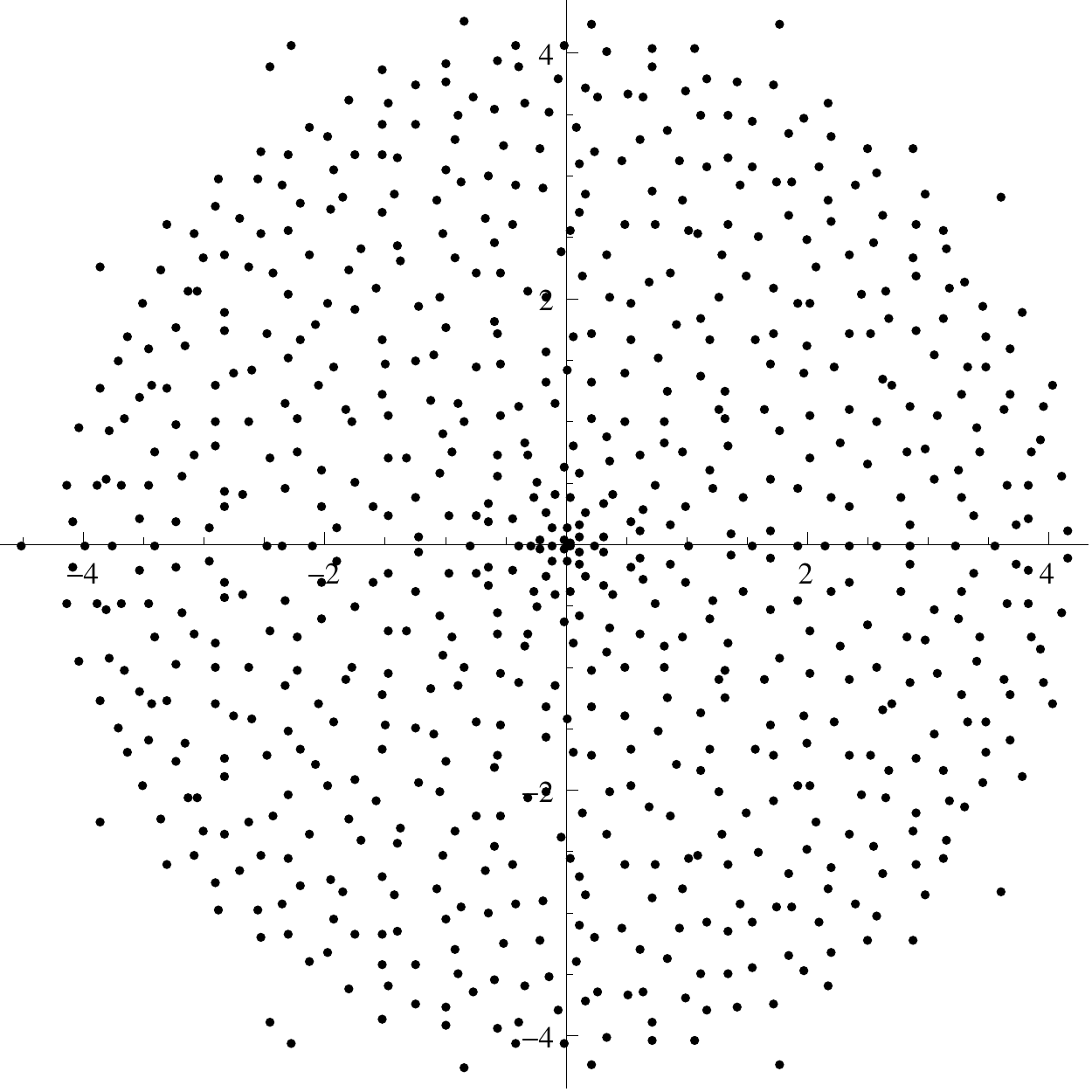}};
\end{tikzpicture}
\end{minipage} 
\caption{Roots of a random polynomial $p_{1000}(z) = \sum_{k=1}^{1000}{\gamma_k\cdot z^k (k!)^{-1/4}}$ (left) and the roots of the polynomial arising from differentiating 250 times (right).} 
\end{figure}

One natural question is now whether the dynamics that can be observed for random Taylor polynomials is in any way universal. We refer to Figure 2 for an example:  after differentiating
a random polynomial of a particular type 250 times, we observe a clustering of roots around the origin that is quite similar to that of random Taylor polynomials (or, in any case, seems
to have a similar scaling).
Our main result deals with linear stability for perturbations around the solution for random Taylor polynomials. More precisely, we will consider a small perturbation around the constant function $\psi(t,x) =  1 + w(t,x),$
where we assume that $\| w(t,x) \|_{L^{\infty}} \ll 1$. The linearized evolution can be derived from
\begin{align*}
 \frac{\partial w}{\partial t} &=  \frac{\partial}{\partial x}\left(  \left(\frac{1}{x} \int_0^x \psi(y) dy\right)^{-1} \psi(x)\right)\\
&=  \frac{\partial}{\partial x}\left(   \left(1 + \frac{1}{x} \int_0^x  w(y) dy\right)^{-1} (1+w(x))\right)\\
&=  \frac{\partial}{\partial x}\left(  \left(1 - \frac{1}{x} \int_0^x  w(y) dy\right) (1+w(x))\right) + \mbox{l.o.t.}\\
&=   \frac{\partial w}{\partial x} - \frac{\partial }{\partial x} \frac{1}{x} \int_0^x  w(y) dy  + \mbox{l.o.t.}
\end{align*}
and is thus given by
\begin{equation} \label{linear}
 \frac{\partial w}{\partial t} = \frac{\partial w}{\partial x} - \frac{\partial }{\partial x} \frac{1}{x} \int_0^x  w(y) dy. 
 \end{equation}
Our main result is that the constant solution $$\psi(t,x) = \chi_{0 \leq x \leq 1-t}$$
 (the dynamics of random Taylor
polynomials) has linear stability for small perturbations close to the origin.
\begin{thm}
If the perturbation $w(0,x)$ is compactly supported and has mean value 0, then the linearized evolution $(\ref{linear})$ is an $L^2-$contraction, i.e.
$$ \frac{\partial}{\partial t}\big|_{t=0} \int_0^{\infty}{w(t,x)^2 dx} \leq 0.$$
\end{thm}
One interesting aspect of this result is that it follows from a variation of a classical Hardy inequality stating if $f:(0,\infty) \rightarrow \mathbb{R}$ is measurable, then
$$ \int_{0}^{\infty}{\frac{f(x)}{x^2} \left( \int_0^x f(y) dy\right) dx} \leq \int_{0}^{\infty}{\frac{f(x)^2}{x} dx},$$
where the inequality is strict unless $f \equiv 0$. (We interpret the inequality in the usual sense: if the right-hand side is finite, then so is the left-hand side and it is smaller.)
Inequalities of this flavor have
been studied intensively in their own right \cite{bee, hardy0, hardy11, hardy1,hardyother1, hardy2, hardy3, hardy4}.

\subsection{Open Problems.} Since the purpose of this paper is to only pose the question and present an evolution equation from a mean field approximation, there are many open problems.

 \textit{1. Rigorous derivation.} 
 In Section \ref{sec:derivation}, we assume the existence of an
underlying limiting dynamics and also assume radial dynamics to be dominant. 
A rigorous derivation requires a more detailed understanding of the local behavior in time. In the one-dimensional
case, this is an old problem: do the roots of polynomials (all of whose roots are on the real line) become more
regularly distributed under differentiation? In particular, do they obey a regular spacing at a local scale?
We refer to \cite{farmer, farmer2, feng, pem2} and references therein.\\
\textit{2. Qualitative behavior.} Certainly one of the most appealing aspects of having a partial differential equation describing the
limiting behavior is the ability to understand and analyze the asymptotic behavior of polynomials with a large
number of roots. What can be said about typical solutions? Is there a way to move between the partial differential equation
and families of polynomials in a way where they mutually inform each other?\\
 \textit{3. Breakdown of symmetry.} Something that we intrinsically assume in our derivation is that radial 
distributions stay radial under differentiation; however, slight deviations from radial symmetry stemming from the fact that we only have finitely many points might get amplified. Radial stability is a question
that might already be interesting and accessible from a mean field perspective.\\
\textit{4. Special cases.} Are there families of polynomials with radial roots whose roots exhibit particularly interesting dynamical
behavior under differentiation? We are only aware of the random Taylor polynomials: are there others?\\
\textit{5. Mixed Real/Complex Behavior.} One question one could ask is to understand the behavior of the real
roots in the presence of complex roots (which, naturally, also evolve). Is the corresponding equation simpler?\\
 \textit{6. Attractors.} Do generic distributions converge to the behavior of random Taylor polynomials in a neighborhood
 of the origin? More precisely, since the evolution of roots seems to be driven by those roots with smaller absolute
 value, it is conceivable that many cases develop a universal profile close to the origin after short time (perhaps behaving
 like random Taylor polynomials)?
Can the stability analysis of Random Taylor polynomials be
further refined? \\
We emphasize question \textit{7. the Non-Radial Case}. The `Lagrangian perspective' (following an individual fluid
particle, here: a single root) is easy enough to describe: if we are given a density of roots $\mu$,
then locally at a point $z$, the roots can be thought of as moving in a direction given by the
$$ \mbox{Cauchy-Stieltjes transform} \qquad - \left( \int_{\mathbb{C}}{ \frac{d \mu(y)}{z-y}} \right)^{-1}.$$
Our equation can be understood as a simplification of the Cauchy-Stieltjes transform in the radial setting. Can
this intuition be made precise? What can be said about the dynamics of this more general equation?
We conclude with an interesting
question that becomes only relevant in the non-radial setting.\\
\textit{8. Which root vanishes?} Suppose we are given a polynomial $p_n: \mathbb{C} \rightarrow \mathbb{C}$ of degree $n$.
It has $n$ roots, and its derivative $p_n'$ has $n-1$ roots. We like to think of the roots of the derivative $p_n'$ has the roots of
$p_n$ being moved a little bit in a direction prescribed by the Cauchy-Stieltjes transform. This line of reasoning has been
pursued and made rigorous by Williams and the first author \cite{or, or2} who showed various types of \textit{pairing} results
between the roots of $p_n$ and $p_n'$ in the stochastic setting.  However, one of the roots `vanishes' (or, put differently, will be unpaired). It has been
empirically observed that the root that vanishes seems to be close to a root of the Cauchy-Stieltjes transform. The converse is
fairly easy to establish: \textit{if} the Cauchy-Stieltjes transform is large in a point where we have an isolated root, then the
derivative has a root nearby (at scale roughly $\sim n^{-1}$). In our radial setting, the Cauchy-Stieltjes transform vanishes
at the origin which is consistent with numerics. However, a rigorous argument in that direction is missing.

\section{Derivation of the Equation} \label{sec:derivation}
\subsection{Assumptions.} Our derivation is conditional on several assumptions that we discuss here.  They appear to be quite
interesting problems in their own right.\\

\textit{1. The limiting dynamics exists.}  It is not a priori clear that the limiting dynamics exists. This problem naturally comes
in various flavors depending on the strength of the statement. Suppose we are given an initial radial density $2 \pi r u_0(r)$
and sample $n$ points independently at random from that distribution, is it true that the distribution of the $n \cdot t-$th
derivative of the associated polynomials has a distribution of roots that converges weakly to a limiting distribution as $n \rightarrow \infty$?
Different types of statements are conceivable: for example, one could wonder whether, for $n$ sufficiently large and $\varepsilon$
small, it is true that if a polynomial has roots $\varepsilon-$close to some radial distribution, whether the the $n \cdot t-$th derivative
will have the distribution of roots $\delta-$ close to another fixed distribution, where $\delta \rightarrow 0$ as $\varepsilon \rightarrow 0$
as $n \rightarrow \infty$. This would correspond to a more deterministic statement. \\

\textit{A Consequence.} A consequence of the limiting dynamics existing is that radial distributions remain radial. This simply follows
from the fact that, for polynomials, differentiation and rotation around the origin, commute. We can thus phrase a simple way of
disproving Assumption \textbf{1}: it would be false, for example, if it were the case that non-radial dynamics, slight inhomogeneities
in the angular structure, for example, would lead to unstable directions of the dynamics. \\

\textit{2. Roots remain separated.} If the dynamics does indeed exists, it follows that it remains radial by the argument outlined above.
It would then remain to understand how this radial structure evolves. We perform a corresponding computation under the assumption
that roots do not `coalesce'. When sampling $n$ roots from a smooth, compactly supported, radial distribution, we expect each root to be distance $\sim n^{-1/2}$
from its nearest neighbors (on average). A typical root is thus not too close to its neighbors. Our derivation of the radial movement
is based on decomposing the logarithmic potential into a near-field and a far-field: if the neighbors are all far away, the near-field does not
have sufficient impact and the far-field dominates. Phrased differently, the way roots move depends on the global structure and not on
the local structure. This assumption seems fairly reasonable, the `root-repulsion' phenomenon is fairly well understood and it is not
difficult to show that this assumption is indeed true after differentiating $k$ times where $k$ is fixed. We refer to \cite{sean, or, or2}.\\

We emphasize that both assumptions above are indeed satisfied for random Taylor polynomials which are correctly described by the
evolution equation we derive.

\subsection{Derivation of the Mean-Field Limit under the Assumptions.} 
In this section we will deduce the mean-field limit equation
$$  \frac{\partial \psi}{\partial t}  = \frac{\partial}{\partial x}\left(  \left(\frac{1}{x} \int_0^x \psi(y) dy\right)^{-1} \psi(x)\right)$$
under the assumptions that the limiting evolution exists and is continuous. Another assumption is that radial measures
remain radial under the evolution (this is clearly the case in the limit but not necessarily clear for large values of $n$, 
see \S 2.3). They key ingredient is the identity
$$ \frac{p_n'(z)}{p_n(z)} = \sum_{k=1}^{n}{\frac{1}{z-z_k}},$$
where $z_1, \dots, z_n$ are the roots of $p_n(z)$. We will now fix the root $z_{\ell}$ and are interested whether there
is a root of $p_n'$ nearby and whether that nearby root can be written in terms of the distribution of the roots of $p_n$. If a nice limiting
density exists, we can assume that the $n$ roots are spread out over area $\sim 1$, this means that the distance from 
a root and its nearby neighbors is typically at scale $\sim n^{-1/2}$. Rotational symmetry allows us to assume $z_{\ell} \in \mathbb{R}$.
Roots of $p_n'$ nearby satisfy the equation
$$ z_{\ell} -z = \left(  \sum_{k=1 \atop k \neq \ell}^{n}{\frac{1}{z-z_k}}\right)^{-1}.$$
It remains to estimate the size of the sum under the assumption that the roots are distributed according to a density $\psi(t,x)$.
To this end, we consider the following integral for two parameters $r,s > 0$
$$ \frac{1}{2\pi} \int_{0}^{2\pi} \frac{1}{r - s e^{it}} dt = \begin{cases} 0 \qquad &\mbox{if}~r < s\\
1/r \qquad &\mbox{if}~r > s \end{cases}.$$
If the number of roots at distance $x$ from the origin is given by $n \psi(t,x) dx$, then this integral (see Fig. 2) suggests that 
$$ z_{\ell} - z \sim  \left(  \sum_{k=1 \atop k \neq \ell}^{n}{\frac{1}{z_{\ell}-z_k}}\right)^{-1} \sim n^{-1} \left( \int_{0}^{z_{\ell}} \frac{\psi(t,x)}{z_{\ell}} dx\right)^{-1}.$$
This suggests that the root moves to the left by a factor only determined by this integral over the density of roots having a smaller norm.
\begin{center}
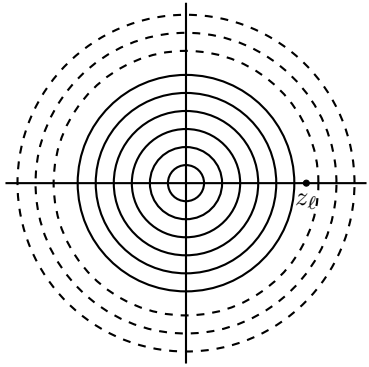
\begin{figure}[h!]
\begin{tikzpicture}[scale=0.8]
\draw [thick] (-3, 0) -- (3,0);
\draw [thick] (0,-3) -- (0,3);
\filldraw (2,0) circle (0.05cm);
\node at (2, -0.3) {$z_{\ell}$};
\draw [thick] (0,0) circle (0.3cm);
\draw [thick] (0,0) circle (0.6cm);
\draw [thick] (0,0) circle (0.9cm);
\draw [thick] (0,0) circle (1.2cm);
\draw [thick] (0,0) circle (1.5cm);
\draw [thick] (0,0) circle (1.8cm);
\draw [thick, dashed] (0,0) circle (2.2cm);
\draw [thick, dashed] (0,0) circle (2.5cm);
\draw [thick, dashed] (0,0) circle (2.8cm);
\end{tikzpicture}
\caption{Only the inner circles contribute to the integral.}
\end{figure}
\end{center}
Phrased differently, roots at distance $x$ from the origin (whose total density is given by $\psi(t,x)$) move to slightly smaller distance with a speed determined by 
$$ \mbox{the nonlocal vectorfield} \sim  - \left( \frac{1}{x} \int_{0}^{x} \psi(t,y) dy\right)^{-1}.$$
We recall that if we are on the real line and have a density $\phi$ subjected to a vector field $v$, then the evolution equation is a special case of a convection-diffusion equation and reads
$$ \frac{\partial}{\partial t} \phi = - \frac{\partial}{\partial x} \left( v \phi\right)$$
which results in the desired equation.

\section{Proof of Linear Stability}
We conclude the paper by establishing linear stability. We start with a dual version of the generalized Hardy inequality. Given the substantial previous works on the subject,
it is presumably stated somewhere in the literature, however, it is also rather easy to deduce from the generalized form of the Hardy inequality. We understand the inequality in the usual sense (i.e. if the right-hand side is finite, then so is the left-hand side which is then
also bounded by the right-hand side).

\begin{lemma}[Hardy-type inequality] Let $f:(0,\infty) \rightarrow \mathbb{R}$ be measurable. Then
\begin{equation}
\int_{0}^{\infty}{\frac{f(x)}{x^2} \left( \int_0^x f(y) dy\right) dx} \leq \int_{0}^{\infty}{\frac{f(x)^2}{x} dx},
\end{equation}
where the inequality is strict unless $f \equiv 0$. 
\end{lemma}
\begin{proof} We start with an application of Cauchy-Schwarz 
\begin{align*}
\int_{0}^{\infty}{\frac{f(x)}{x^2} \left( \int_0^x f(y) dy\right) dx} &= \int_{0}^{\infty}{\frac{f(x)}{\sqrt{x}}  \frac{1}{x^{3/2}}\left( \int_0^x f(y) dy\right) dx} \\
&\leq \left(  \int_{0}^{\infty}{\frac{f(x)^2}{x} dx} \right)^{1/2} \\
&\cdot \left( \int_{0}^{\infty}  \frac{1}{x}\left(\frac{1}{x} \int_0^x f(y) dy\right)^2 dx\right)^{1/2}.
\end{align*}
It thus suffices to prove the inequality 
$$ \int_{0}^{\infty}  \frac{1}{x}\left(\frac{1}{x} \int_0^x f(y) dy\right)^2 dx \leq \int_{0}^{\infty}{\frac{f(x)^2}{x} dx}.$$
This inequality, however, is known as the generalized Hardy inequality (see \cite{bee}). Indeed, a more general result is given by
Theorem 330 in the book of Hardy-Littlewood-Polya \cite{hardy} implying for $p > 1$ and $r > 1$ that
$$ \int_{0}^{\infty} x^{-r} \left( \int_{0}^{x}{f(y) dy}\right)^p dx \leq \left( \frac{p}{r-1}\right)^p \int_{0}^{\infty}{ x^{-r} (x f(x))^p dx}.$$ 
Our case is merely $(p,r) = (2,3)$. 
\end{proof}

\begin{proof}[Proof of the Stability Statement.]
The linearized equation is given by
$$ \frac{\partial w}{\partial t} = \frac{\partial w}{\partial x} - \frac{\partial }{\partial x} \frac{1}{x} \int_0^x  w(y) dy. $$
We try to understand the evolution of
\begin{align*}
\frac{\partial}{\partial t} \frac{1}{2}\int_{0}^{\infty}{ w(t,x)^2 dx} &= \int_{0}^{\infty}{ w(t,x) \frac{\partial w(t,x)}{\partial x} -  w(t,x)\frac{\partial }{\partial x} \frac{1}{x} \int_0^x  w(t,y) dy dx}\\
&= \frac{1}{2} \int_{0}^{\infty}{ \left( \frac{\partial}{\partial x} w(t,x)^2 \right) dx} \\
&- \int_{0}^{\infty}{ w(t,x)\left(\frac{\partial }{\partial x} \frac{1}{x} \int_0^x  w(t,y) dy\right) dx}.
\end{align*}
However, since $w(0, \cdot)$ is compactly supported, the first term at time $t = 0$ reduces to
$$ \frac{1}{2} \int_{0}^{\infty}{ \left( \frac{\partial}{\partial x} w(0,x)^2 \right) dx} = -\frac{w(0,0)^2}{2} \leq 0.$$
We will disregard this quantity and focus on the second term.
The product rule and the fundamental theorem of calculus result in
$$ \frac{\partial }{\partial x} \frac{1}{x} \int_0^x  w(t,y) dy =-\frac{1}{x^2} \int_0^x  w(t,y) dy + \frac{w(t,x)}{x}.$$
It thus remains to understand the sign of the integral which, due to the Hardy inequality proven above, is given by
$$  \int_{0}^{\infty}{ \left( \frac{w(t,x)^2}{x} - \frac{w(t,x)}{x^2} \int_0^x  w(t,y) dy\right) dx} \geq 0.$$
\end{proof}

\end{document}